\documentclass[leqno,12pt]{article}

\usepackage{epsfig}

\usepackage{amsmath}
\usepackage{amscd}
\usepackage{amsopn}
\usepackage{amsthm}
\usepackage{amsfonts,amssymb,bbm}
\usepackage{latexsym}

\usepackage{color}

\makeatletter
\@addtoreset{equation}{section}
\makeatother

\newtheorem{lemma}{LEMMA}[section]
\newtheorem{proposition}[lemma]{PROPOSITION}
\newtheorem{corollary}[lemma]{COROLLARY}
\newtheorem{theorem}[lemma]{THEOREM}
\newtheorem{remark}[lemma]{REMARK}
\newtheorem{remarks}[lemma]{REMARKS}

\newtheorem{examples}[lemma]{EXAMPLES}

\newcommand{\nat}{\mathbbm{N}}

\renewcommand{\a}{\alpha}
\renewcommand{\b}{\beta}

\newcommand{\g}{\gamma}
\newcommand{\vp}{\varphi}
\newcommand{\ve}{\varepsilon}

\newcommand{\reald}{{\mathbbm R^d}}

\newcommand{\on}{\quad\text{ on }}
\newcommand{\und}{\quad\mbox{ and }\quad}
\newcommand{\inv}{^{-1}}
\newcommand{\ov}{\overline}

\newcommand{\W}{\mathcal W}  
\newcommand{\C}{\mathcal C}  

\newcommand{\F}{\mathcal F}

\renewcommand{\H}{{\mathcal H}}
\newcommand{\B}{\mathcal B}

\newcommand{\M}{\mathcal M}

\newcommand{\U}{{\mathcal U}}

\newcommand{\itemframe}%
{\setlength{\parskip}{10pt}\begin{enumerate} \setlength{\topsep}{10pt}%
\setlength{\itemsep}{15pt}\setlength{\parsep}{5pt}}

\newcommand{\mx}{\mu_x}
\newcommand{\my}{\mu_y}

\newcommand{\vy}{\ve_y}
\newcommand{\vx}{\ve_x}
\newcommand{\vz}{\ve_z}

\newcommand{\kap}{\operatorname*{cap}}
\newcommand{\tkap}{\widetilde{\operatorname*{cap}}}

\newcommand{\Uo}{\mathcal U_0}

 \title{Scaling invariant Harnack inequalities\\ in a general setting}
\author{Wolfhard Hansen and Ivan Netuka} 
\date{}

\begin{document}
\maketitle

\begin{abstract}
In a setting, where only ``exit measures'' are given,  as they are  associated with an arbitrary right continuous
 strong Markov process on a separable metric space, we provide simple criteria for 
the validity of Harnack inequalities for positive harmonic functions.
These inequalities are   scaling invariant
with respect to a metric on the state space which, having an associated Green function, 
may  be adapted to the special situation. In many cases, this also implies continuity of harmonic functions 
and H\"older continuity of bounded harmonic functions. 
The results apply to large classes of L\'evy (and similar) processes.


\end{abstract}

\section{Overview}

The study of Harnack inequalities for positive  functions which are harmonic
with respect to rather general partial differential operators of second order, diffusions respectively,
has a long history (see \cite{kassmann-history} and the references therein).  
Fairly recently,   during the last 15 years,  Harnack inequalities 
have  been investigated for harmonic functions with respect to  various classes 
of discontinuous Markov processes, integro-differential operators respectively
(see \cite{bass-levin,bass-kassmann,bogdan-stos-sztonyk,grzywny,HN-harnack,kass-mim-jump,%
kass-mim-final,kim-mimica,mimica-levy,mimica-harnack,mimica-harmonic,rao-song-vondracek,%
sikic-song-vondracek,song-vondracek-1,song-vondracek-2,vondracek}). 

The aim of this paper  
is  to offer  a~very general  analytic approach  to scaling invariant Harnack inequalities for positive  
universally measurable functions on a~separable metric space $X$  which are
harmonic with respect to  given ``harmonic measures'' $\mu_x^U$ (not charging $U$), 
$x\in U$, $U$~open in $X$ (see (\ref{har-def})). 
For $V\subset U$ the corresponding measures are supposed to be compatible in a  way
 which is obvious for  exit distributions of right continuous strong Markov processes 
and harmonic measures on  balayage spaces (see Examples \ref{hunt-bal}). 
An additional ingredient we shall need is a~``quasi-capacity''  on $X$ having suitable scaling properties
such that an estimate of Krylov-Safonov type holds (see (\ref{Kry-Saf})). 

Then a certain property (HJ)
of the measures $\mu_x^U$, which in Examples \ref{hunt-bal} trivially holds for diffusions and harmonic spaces,
is necessary and sufficient for the validity of scaling invariant Harnack inequalities (Theorem \ref{m-harnack}).
For L\'evy processes 
it is easy to specify simple  properties of the L\'evy measure which imply (HJ) for the exit distributions
(see, for example, Lemma \ref{ls-suff}). 

In Section \ref{Green-Harnack}, we discuss properties of an associated ``Green function'' which 
allow us to prove a Krylov-Safonov estimate for the corresponding capacity. This leads to 
Theorem~\ref{G-harnack} and, using recent results on H\"older continuity from~\cite{H-hoelder}, to
Theorem~\ref{G-continuous} on (H\"older) continuity of harmonic functions.

After a first application to L\'evy processes (Theorem \ref{appl-levy}) we   discuss 
consequences of  an estimate of Ikeda-Watanabe type (Theorems \ref{IW-harnack} and \ref{iso}).

In a final Section \ref{section-intrinsic}, we indicate how a Green function satisfying (only) a weak 3G-property
 leads to Harnack inequalities which are scaling invariant with respect to an intrinsically
defined metric.

\section{Harmonic measures and harmonic functions}\label{setting}

Let $(X,\rho)$ be a separable metric space. In fact, the separability will only be used to ensure   
that  finite measures~$\mu$ on its $\sigma$-algebra $\B(X)$ of Borel subsets satisfy 
 \begin{equation}\label{tight}
 \mu(A)=\sup\{\mu(F)\colon F \mbox{ closed, } F\subset A\}, \qquad A\in \B(X)
 \end{equation} 
(recall that  every finite measure on the completion of $X$  is tight). 

For every open set $Y$ in $X$, let $\U(Y)$ denote the set of all open sets $U$ such that the closure
$\ov U$ of $U$ is contained in $Y$.
Given a~set $\F$ of numerical functions on~$X$, let~$\F_b$, $\F^+$ be the set of all functions in $\F$
which are bounded, positive respectively.
Let $\M(X)$ denote the set of all finite measures on $(X,\B(X))$
(which we also consider as measures on the $\sigma$-algebra $\B^\ast(X)$ of all
universally measurable sets).  For every  $\mu\in\M(X)$, let $\|\mu\|$ denote the total mass $\mu(X)$.

For sufficient flexibility in applications, we consider harmonic measures only for open
sets which are contained in a given open set $X_0$ of $X$. More precisely,
we suppose that we have    measures~$\mu_x^U\in\M(X)$, $x\in X$,  $U\in \U(X_0)$,  such that the following hold
for all $x\in X$ and $U,V \in \U(X_0)$ (where $\vx$ is  the  Dirac measure at~$x$):

{\it
\begin{itemize}
\item[\rm(M$_0$)]
 The measure $\mu_x^U$ is supported by $U^c$, and $\|\mu_x^U\| \le 1$. If $x\in U^c$, then $\mu_x^U=\ve_x$.
\item[\rm(M$_1$)]
The  functions $y\mapsto \mu_y^U(E)$, $E\in\B(X)$, are universally measurable on $X$ and
\begin{equation}\label{it-bal} 
\mu_x^U=(\mu_x^{V})^U := \int \mu_y^U\,d\mu_x^{V}(y), \qquad \mbox{ if } V\subset U.
\end{equation} 
\end{itemize} 
}

Of course, stochastic processes and  potential theory  abundantly provide examples (with $X_0=X$).

\begin{examples}\label{hunt-bal}{ \rm
1. Right  process $\mathfrak X$ with strong Markov property on a Radon space $X$ 
and
\[
   \mu_x^U(E):=\mathbbm P^x[X_{\tau_U}\in E],  \qquad E\in \B(X),
\]
where $  \tau_U:=\inf\{t\ge 0\colon X_t\in U^c\}$  
 (\cite[Propositions 1.6.5 and 1.7.11, Theorem 1.8.5]{beznea-boboc}).

If $U,V\in \U(X_0)$ with $V\subset U$, then $\tau_U=\tau_V+\tau_U\circ \theta_{\tau_V}  $,
and hence, by the strong Markov
property, for all $x\in X$ and $E\in \B(X)$, 
\begin{equation*}  
\mx^U(E)=\mathbbm P^x[X_{\tau_U}\in E]=\mathbbm E^x\left(\mathbbm
P^{X_{\tau_V}} [X_{\tau_U}\in E]\right)=\int \my^U(E)\,d\mx^V(y). 
\end{equation*}

2. Balayage space $(X,\W)$ (see {\rm \cite{BH,H-course}})  such that $1\in \W$,  
\[
               \int v\,d\mu_x^U= R_v^{U^c}(x):=\inf\{ w(x)\colon w\in \W,\ w\ge v \mbox{ on } U^c\}, \qquad v\in \W.
\] 
The properties {\rm(M$_0$)} and {\rm(M$_1$)} follow from \cite[VI.2.1, 2.4, 2.10, 9.1]{BH}. 
}
\end{examples}

Going back to the general setting,  let us consider  $x\in X$ and $U,V\in \U(X_0)$ such that $V\subset U$,
and note first 
that, having (M$_0$),  equality (\ref{it-bal}) amounts to the equality
\begin{equation}\label{it-bal-1}
\mu_x^U = \mu_x^{V}|_{U^c} + \int_{U}  \mu_y^U\,d\mu_x^{V}(y)
\end{equation} 
showing that $\mx^U\ge \mx^V$ on $U^c$. However, 
\begin{equation}\label{monoto}
\|\mu_x^U\| =\| (\mu_x^{V})^U\|  = \int \|\mu_y^U\|\,d\mu_x^{V}(y)\le \|\mx^V\|.
\end{equation} 
In particular, for any $A\in \B(X)$ containing  $U$,
\begin{equation}\label{reverse}
       \mu_x^U(A^c)\ge \mu_x^V(A^c) \und  \mu_x^U(A) \le \mu_x^V(A).
\end{equation}

For every $U\in \U(X_0)$, let  $\H(U)$ denote the set of all universally measurable real  functions~$h$ on~$X$ 
which are \emph{harmonic on $U$}, that is, such that, for  all  $V\in \U(U)$ and $x\in V$,
the function $h$ is $\mu_x^V$-integrable and 
\begin{equation}\label{har-def}
                     \int h\,d\mu_x^V=h(x).
\end{equation}
It is easily seen that, for  all bounded universally measurable functions $f$ on $X$ and $U\in \U(X_0)$, the function
\begin{equation}\label{ex-harmonic}
h\colon x\mapsto \int f\,d\mu_x^U
\end{equation} 
is harmonic on $U$. Indeed, it suffices to consider the case $f=1_{E_0}$, $E_0\in \B^\ast(X)$.
Let us fix $U\in \U(X_0)$, $V\in \U(U)$ and  $x\in X$. Then there are $E_1,E_2\in \B(X)$ such that 
$E_1\subset E_0\subset E_2$
and $\mx^U(E_1)=\mx^U(E_2)$.   By (\ref{it-bal}), 
\begin{equation}\label{E12}
                     \mx^U(E_j)=\int \my^U(E_j)\,d\mx^V(y) \qquad\mbox{ for }j=1,2.
\end{equation} 
Hence  $\my^U(E_1)=\my^U(E_2)$ for $\mx^V$-a.e.\  $y\in X$. This implies that the equality~(\ref{E12}) 
holds as well for~$j=0$.

\section{Scaling invariant Harnack inequalities}\label{scaling-invariant}

Let us define 
\begin{equation*} 
       U(x,r):=\{y\in X\colon \rho(x,y)<r\},\qquad x\in X, \, r>0.
\end{equation*} 
Moreover, let  $ R_0(x):=\sup\{r>0\colon \ov{U(x,r)}\subset X_0\}$, $x\in X_0$, and 
\begin{equation*} 
    \U_0:=\{U(x,r)\colon x\in X_0,\, r<R_0(x)\}.
\end{equation*} 

We are interested in  the following Harnack inequalities:
\begin{itemize}
\item[\rm (HI)] 
{\it There  exist $\a\in (0,1)$ and $K\ge 1$ such that,
for \hbox{all~$U(x_0,R)\in \Uo$}, 
\begin{equation*}
\sup h(U(x_0,\a R))\le  K \inf   h(U(x_0,\a R)
\quad \mbox{ for all }
    h\in \H_b^+(U(x_0, R) ).
\end{equation*} 
}
\end{itemize}

Let us immediately note   consequences for arbitrary positive harmonic functions.

\begin{proposition}\label{bounded-unbounded}
Suppose that {\rm (HI)} holds with   $\a\in (0,1)$ and \hbox{$K\ge 1$}.
\begin{itemize}
\item[\rm 1.] Then, for all $U(x_0,R)\in \Uo$   and $h\in \H^+(U(x_0,R))$, 
\begin{equation}\label{strong-harnack}
\sup h(U(x_0,\a R))\le K \inf   h(U(x_0,\a R)).
\end{equation}  
\item[\rm 2.] 
If, for all $U\in \U(X_0)$, the functions in $ \H_b^+(U)$ are
continuous on $U$, then,
 for all $U\in \U(X_0)$, the functions in $ \H^+(U)$ are
continuous on $U$.
\end{itemize}
\end{proposition} 

\begin{proof}   Let $U(x_0,R)\in \Uo$ and $V_r:=U(x_0,r)$, $0<r\le R$.
Further, let  $h\in \H^+(V_R)$ and   $0<R'<R$.

1. By (\ref{ex-harmonic}),  we may  define functions $h_n\in \H_b^+(V_{R'})$, 
$n\in\nat$, by 
\begin{equation}\label{def-hn}
h_n(x) :=\int (h\wedge n)\,d\mu_x^{V_{R'}}, \qquad x\in X. 
\end{equation} 
Then  $\sup h_n(V_{\a R'})\le K\inf h_n(V_{\a R'})$  for every $n\in\nat$ . Clearly, $h_n\uparrow h$ as $n\to \infty$. 
Thus $\sup h(V_{\a R'})\le K\inf h(V_{\a R'})$. Since $R'\in (0,R)$
was arbitrary,  (\ref{strong-harnack}) follows.

2. Since  $h-h_n\in  \H^+(V_{R'})$  for every $n\in\nat$, we
now see, by (\ref{strong-harnack}), that
\[
         h-h_n \le K (h-h_n)(x_0) \on V_{\a R'}.
\]
So   $h_n\to h$   uniformly on $V_{\a R'}$. Therefore $h$ is
continuous on $V_{\a R'}$. 
\end{proof}

Given $c\ge 1$, an increasing positive  function $A\mapsto m(A)$, $A$ universally measurable in
$X_0$, will be called a \emph{quasi-capacity} with constant $c\ge 1$ on $X_0$ if, for all universally measurable
 sets $A,B$ in $X_0$,
\begin{equation*}
        m(A)=\sup\{m(F)\colon  F \subset A, \ F\mbox{
                         closed}\}  \und  m(A\cup B)\le c(m(A)+m(B)).
\end{equation*} 
Clearly, every $\mu\in \M(X)$ (restricted to $X_0$) is a quasi-capacity
on $X_0$, and we note already
now that the capacitary set function $\kap$ which will be defined  in Section \ref{Green-Harnack}
is a quasi-capacity (both with constant $1$).

To obtain suitable criteria for the validity of (HI), we suppose that we have  a~quasi-capacity $m$
on $X_0$, 
  an increasing   continuous  function $m_0\colon (0,\infty)\to (0,\infty) $ 
and  $\a,a,\eta \in (0,1/3)$, $c_0\ge 1$ such that 
  the following translation property~(T), scaling property~(SC)   
and estimate~(KS) of Krylov-Safonov type hold:  
{\it
\begin{itemize} 
 \item[\rm (T)] 
For  every $U(x,r)\in \Uo$, 
 $c_0\inv m_0(r)\le m(U(x,r))\le c_0 m_0(r)$.
\item[\rm(SC)] 
For every $r>0$,  
$a m_0(r)\le m_0(\a r)$,  and   $\lim_{r\to 0}m_0(r)=0$.
\item[\rm (KS)] 
For all $U(x,r)\in \Uo$,
$y\in U(x,\a r)$ and closed sets  $F\subset  U(x, \a r)$,   
\begin{equation}\label{Kry-Saf} 
\mu_y^{U(x,r)\setminus F}(F)\ge \eta \, \, \frac{m(F)} {m(U(x,  r))}\,.
\end{equation} 
\end{itemize} 
}

\begin{remark}  
{\rm
Let us observe that (KS) is much weaker than the property of Krylov-Safonov type we may prove
under rather general assumptions on an associated Green function, which yield that (\ref{Kry-Saf})
holds with $m(U(x,\a r))$ in place of~$m(U(x,r))$ (see Proposition \ref{hit-A}).
}
\end{remark}

Finally, let us consider  the following property  
which, of course,  trivially holds if the measures $\mu_x^{U(x,\a r)}$ are supported by the 
boundary of $U(x,\a r)$. 
{\it
 \begin{itemize} 
  \item[\rm (HJ)] 
There exist $\a\in (0,1)$, $c_J\ge 1$ such that, 
for  all $x\in X_0$ and $0<r<\alpha R_0(x)$,          
\begin{equation}\label{harnack-jump-1}
\mu_x^{U(x,\a r)} \le c_J \mu_y ^{U(x,   r)} \mbox{ on }U(x,r)^c, \qquad y\in U(x,\a^2 r).
\end{equation} 
\end{itemize}
}

Clearly, (HI) implies (HJ), since, for every $E\in \B(X)$ with $E\subset U(x,r)^c$, the 
function $z\mapsto \mu_z^{U(x, r)}(E)$  is harmonic on $U(x, r)$
(see (\ref{ex-harmonic})) and  
$\mu_y^{U(x,\a r)}(E)\le \mu_y^{U(x,r)}(E)$, by (\ref{it-bal-1}). Moreover, again by (\ref{it-bal-1}), 
if (\ref{harnack-jump-1}) holds for some $\a\in (0, 1)$, then 
it holds for every smaller  $\a$. 
By the same argument, (HJ) is equivalent to the following property
(at the expense of replacing $\a$ by $\a/2$).

{\it
\begin{itemize} 
\item[\rm (HJ$^\ast$)] 
There exist $\a\in (0,1)$, $c_J\ge 1$ such that, 
for  all $x\in X_0$ and $0<r<\a  R_0(x)$,         
\begin{equation}\label{harnack-jump-3}
\mu_x^{U(x,\a r)} \le c_J \mu_y ^{U(y, r)} \mbox{ on }U(x,r)^c, \qquad y\in U(x,\a^2 r).
\end{equation} 
\end{itemize}
}

Indeed, suppose that (HJ$^\ast$) is satisfied. Let $x\in X_0$, $0<r<(\a/2)  R_0(x)$ and
 $y\in U(x,(\a/2)^2 r)$. Then $U(y,r/2)\subset U(x,r)$, and hence
\begin{equation*}
     \mu_x^{U(x, \a r/2)}\le \mu_y^{U(y,r/2)} \le \mu_y^{U(x,r)} \on U(x,r)^c.
\end{equation*} 
So (HJ) holds. Similarly for the reverse implication.
 
Our   main result on   scaling invariant  Harnack inequalities 
(assuming the properties  (T), (SC) and (KS)) is as follows.

\begin{theorem}\label{m-harnack}
{\rm (HI)}  holds if and only if {\rm(HJ)} holds.
In particular, {\rm (HI)} holds if the measures $\mu_x^{U(x, r)}$, $U(x,r)\in \Uo$, 
are supported by the boundary of $U(x ,r)$.
\end{theorem}

To prove this result let us assume for the remainder of this section 
  that (HJ) holds. We shall modify our results from \cite{HN-harnack}
(which were inspired by \cite{bass-levin}).
Suppose that (SC) and (KS)   hold with $a_1, \a_1$ and (HJ) holds with $\a_2$.

  We  choose first 
\begin{equation}\label{choice-a}
0<\a\le (\a_1\wedge \a_2)/4, 
\end{equation} 
next $l\in \nat$  with $\a_1^l\le \a$,  
and define $a:=a_1^l$. Then (SC), (KS)  and (HJ) (see (\ref{it-bal-1})) hold with these $\a$, $a$.
Let
\[
  \b:=\frac {\eta a} {4 c c_0^2} \und  \tilde \b:=\frac \b{c_J}. 
\]
Of course,   $\tilde \b\le \b\le 1/4$. We choose $j_0,k_0  \in \nat$ such that  
\begin{equation}\label{j0k0}
a (1+\b)^{j_0}>1, \qquad  \a^{{k_0-1}}<(1-\a)/j_0,
\end{equation} 
and fix
\begin{equation}\label{def-K}
      K\ge   \frac{2cc_0^2(1+\b)}{\eta\tilde \b  a^{k_0+2}}.
\end{equation} 
Let us now fix $x\in X_0$ and $0<R<R_0(x)$,
take $r_0:=\a^{k_0}  R$ and choose $r_n\in (0, r_0)$  with 
\begin{equation}\label{mrj}
m_0(r_n) = (1+\b)^{-n} m_0( r_0), \qquad n\in\nat.
\end{equation} 
 We claim that
\begin{equation}\label{sum-rj}
\sum\nolimits_{n\in\nat} r_n< \a R.
\end{equation} 
Indeed, if $n=i+ kj_0$,  where  $1\le i\le j_0$ and   $k\ge0$,
 then, by (\ref{mrj}), (\ref{j0k0}) and (SC),
\[
m_0 (r_n) <  (1+\b)^{-kj_0} m_0( r_0) \le  a^{ k}m_0( r_0) \le m_0 (\a^{k} r_0),
\]
and hence $r_n<\a^{k} r_0 $. 
So,  by   (\ref{j0k0}),
\begin{equation*} 
 \sum\nolimits_{n\in\nat} r_n< j_0 \sum\nolimits_{k=0}^\infty \a^{ k }r_0<\a R.
\end{equation*}

In connection with (\ref{sum-rj}) the following  result will immediately yield  Theorem~\ref{m-harnack}.

\begin{proposition}\label{main-lemma}
  Let  $h\in\H_b^+(U(x_0,R)) $ and 
suppose that $n\in\nat$ and there exists   $x \in U(x_0,2 \a R)$ with
\begin{equation}\label{hbK}
h(x)>(1+\b)^{n-1}K \, \inf h(U(x_0,\a R)).
\end{equation} 
Then  there exists a point $x'\in U(x, r_n)$ such that 
\[
       h(x')>(1+\b) h(x)>(1+\b)^n K \, \inf h(U(x_0,\a R)).
\]
\end{proposition} 

\begin{proof} By (\ref{hbK}), there is  a point $y_0\in U(x_0, \a R)$ such that $h(x)>(1+\b)^{n-1}K h(y_0)$.
Replacing $h$ by  $h/h(y_0)$, if $h(y_0)>0$,  we may assume that $h(y_0)\le 1$.

For every $0<s<R$, let $U_s:=U(x_0,s)$. Moreover, let $r:=r_n$, $R':=(3/4) R$ and
\begin{equation*} 
B:=U(x,r), \ \  B':=U(x,\a r),\ \  B'':=U(x,\a^2 r), \ \  
A:=\{y\in B''\colon h(y)\ge \tilde \b h(x)\}.
\end{equation*} 
 Then $ B \subset U_{4\a R'}\subset U_{\a_2R'}$ and $y_0\in U_ {\a_2 R'}$.
By (KS),  for every closed $E\subset A$,
\begin{equation*} 
 1\ge h(y_0)=\mu_{y_0}^{ U_{R'}\setminus E}   (h) \ge 
  \tilde \b h(x) \mu_{y_0}^{ U_{ R'}\setminus E} (E)\ge   \tilde \b h(x)  \eta {m(E)}/{m(U_{R'} )},
\end{equation*} 
and hence $m(A)< (\eta\tilde \b(1+\b)^{n-1}K)\inv m(U_{R})$.
By (T), (SC)  and (\ref{mrj}),   
\begin{equation*} 
  m(U_{R}) \le c_0    m_0(R) \le c_0 a^{-k_0} m_0(r_0) = c_0 a^{-k_0}(1+\b)^{n} m_0(r).
\end{equation*} 
By   (\ref{def-K}), we therefore conclude that
\begin{equation}\label{cmA}
2c   m(A) <  \frac{2c  m(U_{R})}{\eta\tilde \b (1+\b)^{n-1}K} \le  c_0\inv a^2 m_0(r) \le c_0\inv m_0(\a^2 r)\le   m(B'') .
\end{equation} 
Since   $              m(B'') \le c (m(A)+m(B''\setminus A))$, we
obtain that $  m(B''\setminus A)>(2c )\inv m(B'')$, where $ m(B'')\ge a
c_0^{-2} m(B')$, by (T) and (SC).
So there is  a closed set $F$ in $B''\setminus A$ 
with
\begin{equation*}
 m(F)> a (2cc_0^2)\inv    m(B').
\end{equation*}
Let
\[
 \nu:=\mu_x^{ B'\setminus F}.
\]
The   measure $\nu$ is supported by $F\cup (X\setminus B')$.
By (KS), we see that
\begin{equation}\label{nuF}
\nu(F) \ge \eta m(F)/m(B') > \eta a (2 cc_0^2)\inv  =2\b. 
\end{equation} 
Moreover, 
\begin{equation}\label{hhn}
h(x)=\int h\,d\nu \und \int_F h\,d\nu\le \tilde \b h(x) \nu(F)<\b h(x) \nu(F).
\end{equation} 
In particular, $\nu$ is not supported by $F$. 

Next we   claim that the function $H:=1_{B^c} h$ satisfies
$\mu_x^{B'} (H) \le \b h(x)$. 
Indeed, if~not,  then (HJ)  implies that, for every $y\in B''$,
\begin{equation*} 
h(y)=\mu_y^{B}(h) =\mu_y^{B}(H)\ge c_J\inv \mu_x^{B'}(H)> c_J\inv \b h(x)= \tilde \b h(x),
\end{equation*} 
contradicting the fact that $A\ne B''$, by (\ref{cmA}). Therefore, by (\ref{it-bal-1}),
\begin{equation*} 
\int_{  B^c} h\,d\nu=\nu(H) 
\le \mu_x^{B'}(H)\le \b h(x).
\end{equation*} 

Defining $b:=\sup h(B)$ we have
\begin{equation*} 
\int_{B\setminus B'}  h\,d\nu\le b\nu(B\setminus B') \le b (1-\nu(F))
\end{equation*} 
(where we used that $\|\nu\|\le 1$). Hence, by  (\ref{hhn}),
\begin{equation*} 
h(x)=  \int_F h\,d\nu   +\int_{X\setminus B'} h \,d\nu   
\le  \b h(x)\nu(F)+\b h(x)+ b(1-\nu(F)).
\end{equation*} 
Since  $\nu(F)> 2\b$, by (\ref{nuF}), we thus conclude that
\begin{equation*} 
      b \ge \frac{1- \b-  \b \nu(F)}{1-\nu(F)} \, h(x)>  (1+\b)h(x) > (1+\b)^n K
\end{equation*} 
completing the proof (we have $(1+\b)(1-\nu(F))=1+\b-\nu(F)-\b\nu(F) $).
\end{proof}

\begin{proof}[Proof of Theorem \ref{m-harnack}]     
Let    $h\in \H_b^+(U(x_0,R))$ and $\g:=\inf h(U(x_0,\a r))$. 
If 
\begin{equation*} 
\sup h(U(x_0,\a r))\le K\g
\end{equation*} 
 does not hold, then there exists 
a point   $x_1\in U(x_0,\a R)$ such that $h(x_1)>K\g$.
 Proposition~\ref{main-lemma} and  (\ref{sum-rj}) then recursively lead to
 points $x_2,x_3,\dots$ in  $U(x_0, 2 \a R)$  satisfying 
 $ h(x_n)>(1+\b)^{n-1} h(x)$ for every $n\ge 2$.  This contradicts the boundedness of~$h$.
 \end{proof}

\section{Green function and   Harnack inequalities}\label{Green-Harnack}

In this section, we assume that
for \emph{all} open sets $U$ in  the separable metric space~$(X,\rho)$,  
we have measures $\mu_x^U\in \M(X)$, $x\in X$, such that (M$_0$)
and (M$_1$) hold (see Section~\ref{setting}). 

Let $G\colon X\times X\to (0,\infty]$ be a Borel measurable function 
and, for $\mu\in \M(X)$, let 
\begin{equation*}
G\mu(x):=\int G(x,y)\,d\mu(y), \qquad x\in X.
\end{equation*} 
For every $A\in \B^\ast(X)$, let
  \begin{equation*} 
       \kap A:=\sup\{\|\mu\|\colon \mu\in \M(X), \ \mu(A^c)=0,\
        G\mu\le 1\}. 
\end{equation*} 
Clearly,
\begin{equation*}
\kap A =\sup\{\kap  F \colon F\mbox{ closed set, } F\subset A\},
\end{equation*} 
and the mapping $A\mapsto \kap A $ is increasing and subadditive.
So, in our terminology, $m:=\kap$ is a~quasi-capacity (with constant $1$) on $X$.

As in Section \ref{setting}, we suppose that $X_0$ is an open set in $X$ and $\U(X_0) $
denotes the set of all open sets $U$ with $\ov U\subset X_0$.
For every closed set $A$ in $X$, let
\begin{equation*}
                              \vx^A:=\mu_x^{X\setminus A}.
\end{equation*} 

\begin{remark} {\rm 
For a right process $\mathfrak X$ (see Example \ref{hunt-bal},1), $\vx^A$ is given by
\begin{equation*}
\vx^A(E)=\mathbbm P^x[X_{D_A}\in E],\qquad E\in \B(X),
\end{equation*} 
where $D_A:=\inf \{t\ge 0\colon X_t\in A\}$ denotes the time of the first entry into $A$.
For a balayage space $(X,\W)$ with $1\in \W$ (see Example \ref{hunt-bal},2), the measure $\vx^A$ is the reduced
measure for $x$ and $A$ (see \cite[VI, p.\,67]{BH}).
}
\end{remark}

Let us consider the  following properties.  
{\it 
\begin{itemize} 
\item [\rm (G$_1$)] 
There exists $c_1\ge 1$ such that, for all $U\in \U(X_0)$, $x\in U$, $\delta>0$, 
the following holds: For every closed set~$A\subset U$, 
there exists  a~closed neighborhood $B\subset U$ of~$A$ and
 a~measure $\nu$ on $B$ such that 
\begin{equation}\label{ABnu}
                                 \|\vx^B\|-\delta< c_1\,\|\vx^A\| \und
            \|\vy^A\|\le G\nu(y)\le c_1\,\|\vy^B\| , \ \ y\in X.
\end{equation} 
\item [\rm (G$_2)$] 
There are 
  a~strictly decreasing continuous  function $g\colon
[0,\infty)\to (0,\infty]$ and constants $c,c_D, M_0\in [1,\infty)$,
$\a_0\in (0,1)$  such that, for every $r>0$,
\begin{equation*} 
g(r/2)\le c_D g(r),\quad   M_0 g(r)\le  g(\a_0 r) \und 
 c\inv g\circ \rho\,  \le\,  G \, \le\,  c\, g\circ \rho.
\end{equation*} 
\item[\rm ($G_3$)]   
There exists $c_2\ge 1$ such that, for $x\in X_0$ and $0<r< R_0(x)$, 
\begin{equation}\label{cap-V}
                            \kap U(x,r)\ge c_2\inv g(r)\inv. 
\end{equation} 
\end{itemize} 
}

Let us first note that,  having  (G$_2$),  for every $M\ge 1$, there exists $0<\a_M<1/4$ such that 
\begin{equation}\label{alpha-eta}
                               Mg(r)\le   g(\a_Mr)      \qquad\mbox{ for every }r>0.
\end{equation} 
Indeed,  it suffices to choose $k\in\nat$ with $M\le M_0^k $, $\a_0^k<1/4$, and to take $\a_M:=\a_0^k$.

\begin{remark}{\rm      
Of course, (G$_1$) holds in Example \ref{hunt-bal},2, if
the functions $G(\cdot,y)$ are potentials on~$X$ with superharmonic support~$\{y\}$ and, 
for every continuous real potential~$p$ on~$X$,
there exists a measure~$\nu$ on~$X$ such that $p=G\nu$ (cf.\ \cite{HN-rep-potential}). 
Indeed, let $A$ be a closed set, $A\subset U\in \U(X_0)$ and $\delta>0$. Then,
by \cite[VI.1.2]{BH}, there exists a closed neighborhood $B\subset U$ of $A$ such that 
$\|\vx^B\|-\delta<\|\vx^A\|$. Let $\vp\in \C(X)$, $1_A\le \vp\le 1_B$ and 
\begin{equation*}
 p:=R_\vp:=\inf\{w\in \W\colon w\ge \vp\}.
\end{equation*} 
Then $p$ is a continuous real potential which is harmonic on $X\setminus B$.
Hence there exists a measure $\nu$, which is supported by $B$, such that $p=G\nu$,
and we have $    \|\vy^A\|\le G\nu(y)\le \|\vy^B\|$ for every $ y\in X$.
}
\end{remark}

\begin{proposition}\label{rho-triangle}
Property {\rm (G$_2$)} implies the following. 
\begin{itemize}
\item For every $x\in X$,   $G(x,x)=\lim_{y\to x} G(y,x)=\infty$.
\item The function $G$  has  the \emph{triangle property}: There exists $C\ge 1$, such that 
\begin{equation}\label{triangle}
         G(x,z)\wedge G(y,z)  \le C G(x,y), \qquad x,y,z\in X.
\end{equation} 
\item For all $x\in X$ and neighborhoods $V$ of $x$, $G(\cdot,x)$
is bounded on $V^c$.
\end{itemize} 
\end{proposition} 

\begin{proof} By (\ref{alpha-eta}), $\lim_{r\to 0} g(r)=\infty$. Hence the inequality $c\inv g\circ\rho\le G$
implies that $G(x,x)=\lim_{y\to x} G(y,x)=\infty$. 
    Moreover, if $x,y,z \in X$, then $\rho(x,z)\ge \rho(x,y)/2$ or  $\rho(z,y)\ge \rho(x,y)/2$.
So $G(x,z)\le c^2 c_D G(x,y)$ or $G(y,z)\le c^2 c_D G(x,y)$.
Finally, the last property is satisfied, since $G(y,x)\le c g(r)$ if $\rho(y,x)\ge r$. 
\end{proof} 

\begin{remark}{\rm
In Section \ref{section-intrinsic}, we shall see that, conversely, the properties of $G$ 
stated in Proposition \ref{rho-triangle}  enable the construction of  a~metric $\tilde \rho$ satisfying (G$_2$)
with $\tilde g(r)=r^{-\g}$ for some $\g\ge 1$. 
}
\end{remark}

\begin{lemma}\label{AnuB} 
\begin{itemize}
\item[\rm 1.]  For every $r>0$, $\kap U(x,r) \le c g(r)\inv $.
\item[\rm 2.] If  $A\in \B^\ast(X)$, $\nu\in \M(X)$ and  $G\nu\ge 1$ on $A$, then  $\kap A\le c^2 \|\nu\|$.  
\end{itemize} 
\end{lemma} 

\begin{proof} Let $\mu\in \M(X)$, $G\mu\le 1$.
 If $\mu(U(x,r)^c) =0$, then 
\begin{equation*} 
 c\inv g(r) \|\mu\|\le \int G(x,y)\,d\mu(y)=G\mu(x)\le 1.
\end{equation*} 
If $A\in \B^\ast(X)$, $\mu(A^c)=0$ and $\nu\in\M(X)$ with $G\nu\ge1$ on $A$, then 
\begin{equation*} 
  \|\mu\|\le \int G\nu \,d\mu\le c^2 \int G\mu\,d\nu\le c^2 \|\nu\|.
\end{equation*} 
\end{proof}

Let us introduce the following property.
{\it
\begin{itemize}
\item[\rm ($\ov G_3$) ] 
There 
 exists    $c_2\ge 1$ such that, for  $x\in X_0$ and $0<r<  R_0(x)$,
\begin{equation}\label{RV}
    \|\vy^{\ov{U(x,r)}}\|\ge c_2\inv  g(r)\inv G(y,x), \qquad y\in U(x,r)^c.
\end{equation} 
\end{itemize} 
}
The next proposition  is of independent interest,
since assuming that we have a~$\mathcal P$-harmonic space where $1$ is superharmonic 
(that is, a balayage space  $(X,\W)$ with $1\in \W$, where the harmonic measures $\mu_x^U$
are supported by the boundary of~$U$), and $G(\cdot,x)$ is a potential which is harmonic
on the complement of $\{x\}$,  we trivially have
\begin{equation*} 
                                \bigl  \|\vy^{\ov{U(x,r)}} \bigr \|\ge c\inv g(r)\inv  G(y,x)  ,\qquad y\in U(x,r)^c.
\end{equation*} 
Its second part will be used in Section \ref{section-intrinsic}. 

\begin{proposition}\label{amounts} Suppose that  {\rm(G$_1$)} and {\rm(G$_2$)} hold. 
\begin{itemize} 
\item[\rm 1.]
 Property {\rm (G$_3$)} implies {\rm($\ov G_3$)}.
\item[\rm 2.]
Suppose that $X_0\ne X$ and {\rm ($\ov G_3$)} holds. Then {\rm (G$_3$)} holds.
\end{itemize} 
\end{proposition}

\begin{proof}
 Let $x\in X_0$, $0<r< R_0(x)$, $z\in U(x,r)^c$ (such a point exists if $X_0\ne X$). 
 By (G$_1$), there exists a~closed neighborhood~$B$ of~$A:={\ov{ U(x,r/4)}}$ 
in~$U(x,r/2)$ and a~measure $\nu$ on  $B$ such that 
\begin{equation*}
                      \|\vy^A\|\le G\nu(y)\le c_1\,\|\vy^B\|\qquad \mbox{ for every } y\in X.
\end{equation*} 
Then $\|\nu\|\le c_1\kap U(x,r)$ and, by Lemma \ref{AnuB},2,  $\kap A\le c^2 \|\nu\|$.
Let $s:=\rho(x,z)$. Obviously,  $s/2\le \rho(z,\cdot)\le 2s$ on $B$, and hence
\begin{equation*}
c g(s/2) \|\nu\|\ge G\nu(z) \ge c\inv g(2s) \|\nu\| \ge  (c^2 c_D)\inv G(z,x) \|\nu\|.
\end{equation*} 

1. Assuming (G$_3$), clearly
$
\kap A\ge \kap U(x,r/4)\ge c_2\inv g(r/4)\inv\ge (c_D^2c_2)\inv g(r)\inv,
$
and therefore, using also (\ref{monoto}),
\begin{equation*} 
   c_1 \|\vz^{\ov{U(x,r)}}\| \ge
   c_1 \|\vz^B\|\ge G\nu(z) \ge (c^2 c_D)\inv G(z,x) \|\nu\|\ge 
   (c^4c_D^3c_2)\inv  g(r) \inv G(z,x) . 
\end{equation*} 

2. Assuming ($\ov G_3$),  we have 
 $\|\vz^A\|\ge c_2\inv g(r/4)\inv G(z,x)\ge (c c_D^2c_2)\inv g(r)\inv g(s)$, where
\begin{equation*} 
  \|\vz^A\|\le       G\nu(z)\le c g(s/2)\|\nu\|\le c c_D  c_1g(s)\kap U(x,r),
\end{equation*} 
and hence $\kap U(x,r)\ge (c^2 c_D^3 c_1c_2)\inv  g(r)\inv$.
\end{proof}

\begin{remark}{\rm
The proof of  (2) shows that $(G_3)$ already holds if,
given $x\in X_0$ and  $0<r<R_0(x)/4$,  (\ref{RV}) is satisfied for just one point $y\in U(x,4r)^c$. 
}
\end{remark}

From now on let us assume in this section that (G$_1$), (G$_2$), (G$_3$) hold,
\begin{equation}\label{a-for-KS}
  M:= (2 c_D^2c^2c_1^2) \vee (3c c_2) \und     0<\a\le \a_M 
\end{equation} 
so that $M g(r)\le g(\a r)$ for every $r>0$.   
We intend to prove that we have the properties (T), (SC) and (KS) taking
$m(A):=\kap A$ and $m_0(r):=g(r)\inv$.
 Of course, (T) follows from Lemma \ref{AnuB},1 and (G$_3$). 
If $k\in\nat$ such that $2^{-k}\le \a$, then (SC) holds with $a:=c_D^{-k}$
by the doubling property of $g$. To get (KS)  we first note the following.

\begin{lemma}\label{subset-U}
Let $U$ be an open set in $X$ and $A\subset U$ be  a~closed set.
Then 
\begin{equation*}
\vx^{A\cup U^c}(U)\ge \|\vx^A\|-\sup\nolimits_{z\in U^c}  \|\vz^A\| \qquad\mbox{ for every }x\in X.
\end{equation*} 
\end{lemma} 

\begin{proof} Since $\vx^{A\cup U^c}$ is supported by $A\cup U^c$, we know,  by~(\ref{it-bal}), that
\begin{equation*} 
\|\vx^A\|=\int_{A\cup U^c} \|\vz^A\|\,d\vx^{A\cup U^c}(z).
\end{equation*} 
The proof is completed observing 
that $\vz^A =\vz$, $z\in A$, and $\vx^{A\cup U^c}(U^c)\le 1$. 
\end{proof}

The next proposition establishes    (KS)  
even with $m(U(x, \a r))$ in place of $m(U(x,r))$                         
in (\ref{Kry-Saf}) (where we may remember that, in Example \ref{hunt-bal},1,
we have $\vx^{A\cup U(x_0,  r)^c }(A)=\mathbbm P^x[D_A<\tau_{U(x_0,r)}]$).

\begin{proposition}\label{hit-A}
Let          
$\eta:=(2c_Dc^3c_1^2c_2)\inv$, 
 $U(x_0,r) \in \Uo$,  $x\in U(x_0,\a r)$  and   $A\subset U(x_0,\a r)$ be closed. Then
\begin{equation}\label{hitting-est}
\vx^{A\cup U(x_0,r)^c}  (A)\ge  \eta\,\, \frac {\kap A} {\kap U(x_0,\a r)}\,.
\end{equation} 
\end{proposition} 

\begin{proof} Let $\delta>0$ and $z\in U(x_0,r)^c$. 
  By~(G$_1$), 
there are a closed  neighborhood~$B$ of~$A$, $B\subset U(x_0,\a r)$, 
and a~measure $\nu$ on $B$  such that 
\begin{equation*}
 \|\vx^B\|<c_1(\|\vx^A\|+\delta) \und \|\vy^A\|\le G\nu(y)\le c_1 \|\vy^B\|\mbox{ \ \ for every } y\in X.
\end{equation*} 
Since $\rho(x, \cdot) \le 2\a r$ on~$B$,  we obtain that
\begin{equation*} 
    c_1\inv \|\vx^{B}\|\ge c_1^{-2}\int G(x,\cdot)\,d\nu\ge   (cc_1^2)\inv g(2 \a r) \|\nu\| 
\ge (c_Dcc_1^2)\inv g(  \a r) \|\nu\|.
    \end{equation*} 
Further, $\rho(z,\cdot)\ge r/2$ on $B$ and $ c g(r/2)\le c_D cg(r) \le (2c_Dc c_1^2 )\inv  g(\a r)$,
by (\ref{a-for-KS}). So
\begin{equation*} 
\|\vz^A\|\le \int G(z,\cdot)\,d\nu\le c g(r/2)\|\nu\|\le  (2c_Dcc_1^2)\inv   g(\a r)\|\nu\|.
\end{equation*} 
Combining these two estimates we see that
\begin{equation*} 
       \|\vx^A\| +\delta - \|\vz^A\| \ge  c_1\inv \|\vx^B\|-\|\vz^A\|\ge  (2c_Dcc_1^2)\inv     g( \a r)  \|\nu\|. 
\end{equation*} 
Therefore,  by Lemma  \ref{subset-U}, Lemma \ref{AnuB},2 and  (\ref{cap-V}),
\begin{equation*} 
\vx^{A\cup U(x_0,r)^c} (U(x_0,r))+\delta>  (2c_Dcc_1^2)\inv  g ( \a r ) \|\nu\| \ge \eta\kap A/\kap U(x_0,\a r). 
\end{equation*} 
Since the measure $\vx^{A\cup U(x_0,r)^c}$ is supported by $A\cup U(x_0,r)^c$, the proof is finished.
\end{proof}

By Theorem \ref{m-harnack}, we now obtain the following result.

\begin{theorem}\label{G-harnack}
Suppose that we have {\rm (HJ)} and  $\a\in (0,1)$ which satisfies {\rm (\ref{choice-a})} and {\rm (\ref{a-for-KS})}. 
Then there exists $K\ge 1$ such that, for all  $U(x,R)\in \Uo$, 
\begin{equation*}
     \sup h(U(x,\a R))\le K\inf h(U(x,\a R)), \qquad h\in \H^+(U(x,R)).
\end{equation*} 
\end{theorem}

We recall that, by Proposition \ref{bounded-unbounded},  (HI) and  the continuity of  
all functions in~$\H_b^+(U)$, $U\in \U(X_0)$, imply the continuity
of all  functions  in  $\H^+(U)$, $U\in\U(X_0)$.
In fact, assuming that the constant function $1$ is harmonic on~$X$, \cite[Corollary~3.2]{H-hoelder}   implies 
even the H\"older continuity of all functions  in $\H_b^+(U)$, $U\in \U(X_0)$.   To see this  we only have to verify 
  property (J$_0$) in  \cite{H-hoelder}, that is,    the following.

\begin{proposition}\label{j0} 
There exists $\delta_0>0$ such that, for every $U(x,r)\in \Uo$, 
\begin{equation}\label{j0-formula}
\mu_x^{U(x,\a^2 r)}(U(x,r))>\delta_0.
\end{equation} 
\end{proposition} 

\begin{proof}   
 Let  $U(x,r)\in \Uo$ and $S:= {U(x,\a r)}\setminus U(x,\a^2 r)$. Then, by   (\ref{cap-V}),  Lemma \ref{AnuB},1 and
(\ref{a-for-KS}), 
\begin{eqnarray*} 
         \kap S&\ge& \kap {U(x,\a r)} -\kap U(x,\a^2 r)\\
                    &\ge& c_2\inv g(\a r)\inv -c g(\a^2r)\inv >(2c_2)\inv g(\a r)\inv.
\end{eqnarray*} 
So $\kap F>(2c_2)\inv g(\a r)\inv$ for some  closed set $F\subset S$.
By Proposition \ref{hit-A}, 
\begin{equation*} 
                   \mu_x^{U(x,r)\setminus F}(U(x,r)) =\vx^{F\cup U(x,r)^c}(U(x,r))\ge\eta\kap F/\kap U(x,\a r) \ge (2c c_2^2 )\inv \eta.
\end{equation*} 
To finish the proof we note that $\mu_x^{U(x,\a^2 r)}(U(x,r))\ge \mu_x^{U(x,r)\setminus F}(U(x,r))$,  by (\ref{reverse}). 
\end{proof} 

As already indicated, \cite[Corollary~3.2]{H-hoelder} now yields the following result.

\begin{theorem}\label{G-continuous}
Suppose that {\rm(HJ)} holds and $1\in \H(X)$.
Then there exist $\b\in (0,1)$ and $C\ge 1$
such that, for all $U(x,R)\in \Uo$,           
\begin{equation*}
          |h(y)-h(x)|\le C\|h\|_\infty \left(\frac{\rho(x,y)} R\right)^\b \quad\mbox{ for all } y\in U(x,R), \,h\in\H_b(U(x,R)).
\end{equation*} 
In particular, every universally measurable function,  which is harmonic on an open set $U$ in $X_0$, 
is continuous on $U$.    
\end{theorem}

\section{A first application to L\'evy processes}

In this section, let us assume that $X=\reald$, $d\ge 1$,   $\rho(x)=|x-y|$,  and 
the measures~$\mu_x^U$ are given by a L\'evy process~$\mathfrak X$ on~$\reald$ such that, for some 
Borel measurable function $n\ge 0 $  on $\reald$  and  all $x\in \reald$, $r>0$   and  
  Borel sets~$A$ in~$\reald\setminus \ov {U(x,r)} $,
\begin{equation}\label{levy-system}
\mathbbm P^x[X_{\tau_{U(x,r)}}\in A]=\mathbbm E^x \int_0^{\tau_{U(x,r)}} \int_An(z-X_u)\,dz\,du.
\end{equation}

\begin{lemma}\label{ls-suff}
Suppose that  $0<\a<1/2$ and $c_J\ge 1$ such that,
for     $   y,z\in \reald$,
\begin{equation}\label{KKz}
n(z)\le c_J n(z+  y )\qquad\mbox{ provided \ } |  y|<2\a|z|.
\end{equation} 
Then for all $x\in \reald$, $r>0$ and $y\in U(x,\a r)$,
\begin{equation*} 
                                 \mu_x^{U(x, \a r)} \le c_J
                                   \mu_y^{U(y,\a r)}   \on U(x,r)^c.
\end{equation*} 
In particular, {\rm (HJ)} holds.  
\end{lemma}

\begin{proof} By translation invariance, it suffices to consider the case $x=0$. Let $r>0$
and $y\in U(0,\a r)$, $\tau:=\tau_{U(0,\a r)}$, and let $A$ be a Borel set in $U(0,r)^c$.   
By  (\ref{levy-system})  and translation invariance, 
\begin{eqnarray*} 
 &&\mu_y^{U(y, \a r )}(A)=\mu_0^{U(0,\a  r)}(A-y)\\
&=& \mathbbm E^0 \int_0^{\tau }\int_{A-y}  n(  z-X_u)\,dz\,du
= \mathbbm E^0 \int_0^{\tau  }\int_A
n(  z -X_u+y)\,dz\,du.
\end{eqnarray*} 
If $z\in A$ and  $\tilde z\in U(0,\a r)$, then
$ | y|<\a r< 2\a (1-\a)r<2\a |z-\tilde z|$, and hence $n(z-\tilde z)\le c_J n(z-\tilde z+y)$, by (\ref{KKz}). 
Considering also the case $y=0$, we conclude that 
    $ \mu_0^{ U(0,\a r) }(A)\le c_J \mu_y^{ U(y,\a r) } (A)$.
To complete the proof it suffices to recall that 
$\mu_y^{ U(y,\a r) } \le \mu_y^{ U(0,  r) }$ on $U(0,r)^c$, by (\ref{it-bal-1}).  
\end{proof} 

So, by Theorems \ref{G-harnack} and \ref{G-continuous}, we have the following.   

\begin{theorem}\label{appl-levy} Suppose that  $n$ satisfies {\rm (\ref{KKz})} and that there exists
a Borel measurable function $G\colon \reald\times \reald\to (0,\infty]$ satisfying
{\rm(G$_1$)} -- {\rm(G$_3$)} with $\rho(x,y)=|x-y|$. 
 
Then there exist $\a\in (0,1)$ and  $K\ge 1$ such that, for all $x\in \reald$ and $R>0$,
\begin{equation*} 
     \sup h(U(x,\a R))\le K\inf h(U(x,\a R)) \quad \mbox{ for all } h\in \H^+(U(x,R)).
\end{equation*} 
There are $\b\in (0,1)$ and $C\ge 1$
such that, for  $x\in \mathbbm R^d$ and $R>0$,
\begin{equation*}
          |h(y)-h(x)|\le C\|h\|_\infty \left(\frac{|x-y|} R\right)^\b \quad\mbox{ for all } y\in U(x,R), \,h\in\H_b(U(x,R)),
\end{equation*} 
and every universally measurable  function on $\reald$,  which is harmonic on an open set~$U$ in $X_0$,
is  continuous on $U$.    
\end{theorem}

\begin{remark}{\rm
For a sufficient property which is weaker than (\ref{KKz}) see
(\ref{nyy}).   
}
\end{remark}

\section{Application based on an Ikeda-Watanabe estimate} 

To cover more general  processes let us return to the setting of Section~\ref{Green-Harnack},    
where we have the following: A separable metric space $(X,\rho)$ and       
harmonic measures~$\mu_x^U$ on~$X$, $x\in X$, $U$ open sets in~$X$, 
which satisfy (M$_0$) and  (M$_1$) (see Section \ref{setting}),  
and a~Borel measurable function 
$G\colon X\times X\to (0,\infty]$ such that  (G$_1$) -- (G$_3$) hold.
In particular, we have  an open set $X_0$ in $X$,   
balls  $U(x,r):=\{y\in X\colon \rho(x,y)<r\}$, and
$R_0(x):=\sup\{r>0\colon \ov{U(x,r)}\subset X_0\}$,  $x\in X_0$.

For every $V\in \U(X_0)$,   let $G_V$ be  the associated (Green) function on $V$, that is,  
\[
          G_V(x, y) :=G(x,y)-\int G(z,y) \,d\mu_x^V(z), \qquad x,y\in V,
\]
and $G_V:=0$ outside $V\times V$. 
We suppose that we have the following relation between the functions $G_{U(x,r)}(x,\cdot)$ and 
the harmonic measures  $\mu_x^{U(x,r)}$.
{\it
\begin{itemize}
\item[\rm (IW)]
There exist  $\lambda\in\M(X)$, a kernel $N$ on $X$, $M_{IW}\ge 1$ and $C_{IW}\ge 1$ 
such  that, for all $x\in X_0$, $0<r<R_0(x)$ and Borel sets $E$ in $X\setminus \ov{U(x, M_{IW} r)}$, 
\begin{equation}\label{ike-wat}
C_{IW}\inv   \mx^{U(x, r)}(E)\le \int G_{U(x,r)}(x,z)N(z,E)\,d\lambda(z)\le C_{IW} \mx^{U(x,r)}(E) .
\end{equation} 
\end{itemize}
}

\begin{remark}{\rm       
With $C_{IW} =1$ and $ M_{IW}=1$, (\ref{ike-wat})  is part of the Ikeda-Watanabe formula 
which holds for all (temporally homogeneous) L\'evy processes (see \cite[Example~1 and Theorem~1]{ikeda-watanabe}).

We are indebted to a referee of the manuscript \cite{HN-harnack} (which merged into the present paper) for the hint
that, in the Examples \ref{hunt-bal},2,  the Ikeda-Watanabe formula always holds under mild duality assumptions 
(where $\lambda$
is the Revuz measure of a positive continuous additive functional $H$ given by a L\'evy system $(N,H)$ 
for a  suitable Hunt process~$(X_t)$ and an excessive reference measure $m$ associated with~$(X,\W)$; 
 see \cite{benveniste-jacod, blumenthal-getoor, graversen-rao, revuz}):
\begin{equation*}
                     \lambda (A)= \lim\nolimits_{t\to 0} \frac 1t \, \mathbbm E^m\int_0^t 1_A\circ X_s\,dH_s, \qquad A\in \B(X).
\end{equation*}
}
\end{remark}

We shall get the following results, where it only remains to prove     
that property (HJ) is satisfied  (see (\ref{harnack-jump-1})  and Theorems \ref{G-harnack} and \ref{G-continuous}).

\begin{theorem}\label{IW-harnack}               
Suppose that  there exist   $C\ge 1$ and $\a\in (0,1)$ such that,  for all
$x\in X_0$,  $0<r<\a R_0(x)$,  $y,y'\in U(x,\a r)$,
\begin{equation}\label{Nxy}
                        N(y',\cdot)\le C N(y,\cdot) \on U(x,r)^c
\end{equation} 
 and 
\begin{equation}\label{int-xy}
  \int_{U(x,\a r)} g(\rho(x,z))\,d\lambda(z)\le C\int_{U(y,2\a r)}   g(\rho(y,z))\,d\lambda(z).
\end{equation} 
Then 
scaling invariant Harnack inequalities hold for  functions in $\H^+(U(x,R))$,  \hbox{$x\in X_0$} and $0<R<R_0(x)$.

Moreover, if $1\in \H(X)$, then  scaling invariant H\"older continuity holds for functions in $\H_b(U(x,R))$,  
$x\in X_0$ and $0<R<R_0(x)$, and every  universally measurable function on $X$,
 which is harmonic on
an open set $U$ in $X_0$,  is continuous on~$U$.   
\end{theorem} 

Let us note that (\ref{int-xy}) trivially holds, if $X=\reald$, $d\ge 1$, $\rho(x,y)=|x-y|$   
and~$\lambda$ is Lebesgue measure.

\begin{theorem}\label{iso}
Suppose that $X=\reald$,  $\rho(x,y)=|x-y|$, 
the measure~$\lambda$ in {\rm (IW)} is Lebesgue measure and
there exists   $c_2\ge 1$
such that the normalized Lebesgue measure~$\lambda_{U(x,r)}$ on~$U(x,r)$ satisfies\footnote{The inequality
(\ref{lambda-g}) trivially implies that (G$_3$) holds (see also \cite[(1.14)]{H-liouville-wiener}).}  
\begin{equation}\label{lambda-g}
      G\lambda_{U(x,r)}\le c_2 g(r), \qquad \mbox{$ x\in X_0$, $0<r<R_0(x)$}.
\end{equation} 
 Moreover, assume
 that there are a measure $\tilde \lambda$ on $\reald$, a Borel measurable function    $n\colon \reald\times \reald\to
[0,\infty)$ and $C\ge 1$, $\a\in (0,1)$ 
   such that $N(y,\cdot)=n(y,\cdot)\tilde\lambda$, $y\in X_0$, and,
for all $y,y'\in X_0$ and $\tilde z\in X$,
\begin{equation}\label{nyy}
n(y,\tilde z)\le C n(y',\tilde z), \qquad\mbox{ if }  |y-\tilde z|\ge
|y'-\tilde z| \mbox{ \ and \ } |y-y'|<\a |y'-\tilde z| .
\end{equation} 
Then the conclusions of Theorem \ref{IW-harnack}  prevail.
\end{theorem}

\begin{remark}{ 
Suppose that there exists a
function $n_0\colon [0,\infty)\to [0,\infty)$, $C_0\ge 1$ and $\a\in (0,1)$ such that
 $C_0\inv n_0(|x-y|)\le n(x,y)\le C_0 n_0(|x-y|)$ and
\begin{equation}\label{weak-decreasing}
                         n_0(t) \le C_0  n_0(s), \qquad\mbox{ whenever
                         } 0<s<t<(1+\a) s. 
\end{equation} 
Then  {\rm (\ref{nyy})} is satisfied.
}
\end{remark}

Thus  rather general L\'evy processes abundantly provide  examples for our
approach (see 
\cite{grzywny, kim-mimica, mimica-harnack,  mimica-harmonic, 
rao-song-vondracek, sikic-song-vondracek}).

For the proofs of Theorem \ref{IW-harnack} and \ref{iso} we need  the following simple statement, 
where $M:=2 c_Dc^2$ and $0<\a_M<1/4$ such that $M g(r)\le g(\a_M r)$, $r>0$ 
(see (\ref{alpha-eta})).

\begin{lemma}\label{GGB} Let $y\in X$, $r>0$ and $0<\a<\a_M$. Then
\begin{equation}\label{GGM}
                G_{U(y,r)}(\cdot,y)\ge \frac 12 \, G(\cdot,y) \on U(y,2\a r).
\end{equation} 
\end{lemma}

\begin{proof} 
Let  $x\in U(y,2\a r)$.  Since $G(\cdot,y)\le c g(r) $ on~${U(y,r)^c}$, we obtain that
\begin{equation*} 
\int G(z,y)\,d\mu_x^{U(y,r)}\le  c g(r)       \|\mu_x^{U(y,r)}\| \le  (2cc_D)\inv g(\a r),
\end{equation*} 
whereas $G(x,y)\ge c\inv g(2\a r)\ge (c c_D)\inv g(\a r)$. So (\ref{GGM}) holds.
\end{proof}

\begin{proof}[Proof of Theorem \ref{IW-harnack}] To prove (HJ) we fix $0<\a<(\a_M\wedge M_{IW}\inv)/2$.   
Now let  $x\in X_0$, $0<r<\a R_0(x)$, $y\in U(x, \a^2 r)$, and let  $E$ be a~Borel  set in~$U(x,r )^c$. 
Then $E$ is contained in both $U(x,M_{IW}\a r)^c$ and $U(y,M_{IW}\a r)^c$.  
Hence, by~(\ref{ike-wat}) -- (\ref{int-xy}), Lemma \ref{GGB}  and (\ref{it-bal-1}),  
\begin{eqnarray*} 
\mx^{U(x,  \a^2 r)} (E)&\le& C_{IW} \int G_{U(x,  \a^2 r)}(x,z) N(z,E)\,d\lambda(z)\\
&\le& cCC_{IW} N(y,E)\int_{U(x, \a^2 r)} g(\rho(x,z))  \,d\lambda(z)\\
&\le &c C^2C_{IW}  N(y,E) \int_{U(y, 2 \a^2 r)} g(\rho(y,z))  \,d\lambda(z)\\
&\le &2c^2 C^3C_{IW} \int  G_{U(y,  \a r)}(y,z) N(z,E) \,d\lambda(z) \\
&\le&  2c^2 C^3 C_{IW}^2 \, \my^{U(y,\a r)}(E) \le  2c^2 C^3 C_{IW}^2 \, \my^{U(x, r)}(E). 
\end{eqnarray*} 
Thus (HJ) holds (with $\a^2$ in place of $\a$).
\end{proof} 

\begin{proof}[Proof of Theorem \ref{iso}  {\rm (cf.\ the proof of {\cite[Proposition 6]{grzywny}})}]   
First, we choose $0<\a<(\a_M\wedge M_{IW}\inv)/2$ such that (\ref{nyy}) holds.  
Next, we fix $x\in X_0$, $0<r< R_0(x)$ and a~Borel set    $E$ in $ U(x,r )^c$. 
Then $E\subset X\setminus \ov{U(y,M_{IW}\a r)}$ for every  $y\in U(x, \a^2 r)$. 
 By~(\ref{ike-wat}), (\ref{it-bal-1}) and Lemma \ref{GGB}, 
\begin{eqnarray*} 
\mu_x^{U(x,\a^2 r)}(E)&\le& cC_{IW} \int_{U(x,\a^2 r)}\int_E
g(|x-z|)n(z,\tilde z)\,d\tilde \lambda(\tilde z)\,d\lambda(z),\\
\mu_y^{U(x, r)}(E)  &\ge&\mu_y^{U(y, \a r)}(E)\ge  C_{IW}\inv \int G_{U(y, \a r)}(y,z)N(z,E)\,d\lambda(z)\\
     &\ge&  (2C_{IW}) \inv  \int_{U(y,2\a^2 r)} G(y,z) N(z,E)\,d\lambda(z)\\
&\ge& (2cc_DC_{IW})\inv  g(\a^2 r)\int_{U(x,\a^2 r)}\int_E n(z,\tilde z)\,d\tilde
     \lambda(\tilde z) \, d\lambda(z).
\end{eqnarray*} 
Defining  $\tilde r:=\a^2 r$, we hence  have 
to show that, with some constant $C'>0$,
\begin{equation}\label{suff-tilde}
\int_{U(x,\tilde r)} g(|x-z|)n(z,\tilde z)\,d\lambda(z) \le C'
g(\tilde r) \int_{U(x,\tilde  r)} n(z,\tilde z) \, d\lambda(z)
\end{equation} 
 for every $\tilde z\in U(x,r)^c$.  To that end let us fix $\tilde z\in U(x,r)^c$.

Let $B:=U(x,\tilde r/2)$. Since $g(|x-y|)\le g(\tilde r/2)\le c_D g(\tilde r)$ for every  $y\in B^c$,  
\[ 
\int_{U(x,\tilde r)\setminus B} g(|x-y|)n(y,\tilde z)\,d\lambda(y) \le  c_D g(\tilde r)\int_{U(x,\tilde r)}
n(y,\tilde z) \, d\lambda(y).
\]
Moreover, let
\[
          x':=x+\frac 3 4 \, \frac {\tilde z-x}{|\tilde z-x|}\, \tilde r \und  B':=U(x',\tilde r/4),
\] 
so that $B'\subset U(x,\tilde r)\setminus B$. 
If $y\in B$ and $y'\in B'$, then $|y-\tilde z|\ge |y'-\tilde z|$ and
$|y-y'|<3\tilde r/2< \a( r- \tilde r)<\a |y'-\tilde z|$, 
and therefore, by (\ref{nyy}),
\[
n(y,\tilde z)\le \frac C{ \lambda(B')} \int_{B'} n(y',\tilde z)\, d  \lambda(y')
=\frac {2^d C}{ \lambda(B)} \int_{B'} n(y',\tilde z)\, d  \lambda(y').
\]
Hence
\begin{multline*} 
 \int_B g(|x-y|)n(y,\tilde z)\,d\lambda(y)\\
 \le  2^d C\left(\int_{B'} n(y',\tilde z)\, d  \lambda(y')\right)\cdot
\left(\frac 1{\lambda(B)}\int_B g(|x-y|)\,d\lambda(y)\right),
\end{multline*} 
where, by (\ref{lambda-g}), 
\[\frac1{\lambda(B )} 
\int_Bg(|x-y|)\,d\lambda(y)\le cG\lambda_B(x)\le  cc_2 g(\tilde r/2) \le cc_Dc_2 g(\tilde r). 
\]
Thus (\ref{suff-tilde}) holds with $C':=c_D(1+2^dcc_2C)$. 
\end{proof}

\section{Intrinsic scaling invariant Harnack inequalities}\label{section-intrinsic}

In this section we shall weaken the assumptions 
and prove intrinsic scaling invariant Harnack inequalities, where the metric is derived
from  the Green function.
We start with the same setting as in Section \ref{Green-Harnack} 
(assuming that $X_0$ is a proper subset of $X$) and suppose 
that we have a~Borel measurable function $G\colon X\times X\to (0,\infty]$ 
which satisfies (G$_1$).  Let us define 
\begin{eqnarray*} 
           V(x,s)\!\!\!&:=&\!\!\!\{y\in X\colon G(y,x)\inv<s\}, \qquad x\in X,\, s>0,\\
           S_0(x)\!\!\! &:=&\!\!\! \sup\{s>0\colon \ov{V(x,s)}\subset X_0\}, \qquad x\in X_0.
\end{eqnarray*}  

Instead of (G$_2$) and (G$_3$) we assume the following properties (where also the case
$w=1$ is of interest). 
{\it
\begin{itemize} 
\item [\rm (G$_2'$)]  
For every $x\in X$, 
$G(x,x)=\lim_{y\to x} G(y,x)=\infty$, and there exists 
a~Borel measurable function $w$ on $X$, $0<w\le 1$, such that 
\begin{equation}\label{w-superharmonic}
\int w\,d\mu_x^U\le w(x)\qquad\mbox{ for all open sets $U$ in $X$ and $x\in X$}, 
\end{equation} 
  and  $G$ has the \emph{$(w,w)$-triangle property},
that is, for some    $\tilde c >1$, the function  
\[
\tilde G\colon (x,y)\mapsto \frac{G(x,y)}{w(x)w(y)}
\] 
 satisfies
\begin{equation}\label{ww-triangle}
   \tilde G(x,z)\wedge \tilde G(y,z)\le \tilde c \tilde G(x,y), \qquad x,y,z\in X.
\end{equation} 
Moreover,  $\lambda:=\inf w(X_0)>0$ and, for every $x\in X$ and  neighborhood 
 $V$ of~$x$, the function $G(\cdot,x)/w$ is bounded on $V^c$.
\item[\rm ($\ov G{}_3'$)]
 There exists $c_3\ge 1$ such that, for all  $x\in X_0$ and $0<s<S_0(x)$,
\begin{equation}\label{G3prime}
         \|\vy^{\ov{V(x,s)}}\|\ge c_3\inv s G(y,x), \qquad y\in V(x,s)^c.
\end{equation} 
\end{itemize} 
}
 
Moreover, we introduce the following property. 
{\it
\begin{itemize}
\item[\rm (HJ$'$)]
 There exist $\a\in (0,1)$ and  $c_J\ge 1$ such that, for  all $x\in X_0$ and $0<s<S_0(x)$, 
\begin{equation}\label{ovHJ}
   \mu_x^{\overset\circ V(x,\a s)}\le c_J   \mu_y^{\overset\circ  V(x,s)}\on \overset \circ V(x,s)^c, \qquad y\in \overset\circ V(x,\a^2 s).
\end{equation} 
\end{itemize}
}

\begin{remarks}{\rm
1. Clearly, (G$_2$)  implies (G$_2'$) with $w=1$,   by Proposition  \ref{rho-triangle}. 

2. In interesting cases, the function $G$ does not have the $(1,1)$-triangle property, 
 and hence (G$_2$) cannot hold. However, in these cases the $(w,w)$-triangle property
frequently holds with  $w=  G(\cdot,y_0)\wedge 1$,  where $y_0$ is some fixed point in $X_0$.
This is already the case in classical potential theory and the theory of Riesz potentials, if~$X$ is an 
open ball in $\reald$.
}
\end{remarks}

\begin{proposition}\label{tilde-rho}
There exist a metric $\tilde \rho$ for the topology of $X$, $\g\ge 1$ and $C\ge 1$ 
such that, for all $(x,y)\in X\times X$,
\begin{equation}\label{cGc}
                                 C\inv \tilde \rho(x,y)^{-\g}\le \tilde G(x,y)\le C   \tilde \rho(x,y)^{-\g}.
\end{equation} 
In particular, {\rm (G$_2$)} holds for $\tilde G$ with $\tilde g(r):=r^{-\g}$. 
Moreover, let  
\begin{eqnarray*}
\tilde U(x,r)\!\!\!&:=&\!\!\!\{y\in X\colon \tilde \rho(x,y)<r\}, \qquad x\in X, \,r>0,\\
\tilde R_0(x)\!\!\!&:=&\!\!\!\sup\{r>0\colon \ov{\tilde U(x,r)}\subset X_0\},
\end{eqnarray*} 
and $\b:=(\lambda/C)^{2/\g}$. Then, for all $x\in X_0$ and $r>0$,    
\begin{equation}\label{UVU}  
\tilde U(x,\b r)\subset     V(x,C\inv r^\g)\subset \tilde U(x,r).
\end{equation} 
\end{proposition} 

\begin{proof} 
Since $\tilde G=\infty$ on the diagonal, (\ref{ww-triangle}) implies that $\tilde G(y,x)\le \tilde c \tilde G(x,y)$
and 
\begin{equation*} 
(x,y) \mapsto  \tilde G(x,y)\inv + \tilde G(y,x)\inv
\end{equation*} 
defines a quasi-metric on $X$ which is equivalent to $\tilde G\inv$. So, by  \cite[Proposition 14.5]{heinonen}
(see also \cite[Proposition~6.1]{H-liouville-wiener}), there exist a metric~$\tilde \rho$ on~$X$, $\g\ge 1$ and 
$C\ge 1$ such that  (\ref{cGc}) holds. 

Now let $x\in X$, $r>0$ and $s:=C\inv r^\g$. If $y\in \tilde U(x,\b r) $, then 
\begin{equation*}
G(y,x)\ge \lambda^2 \tilde G(y,x)\ge \lambda^2C\inv \tilde \rho(x,y)^{-\g}>\lambda^2 C\inv (\b r)^{-\g}=s\inv.
\end{equation*} 
Therefore $\tilde U(x,\b r)\subset V(x,s)$.
If $y\in V(x,s)$, then 
\begin{equation*} 
C \tilde \rho(x,y)^{-\g}\ge \tilde G(y,x) \ge G(y,x)> C r^{-\g},
\end{equation*} 
and hence $\tilde \rho(x,y)< r$. So  $V(x, s)\subset \tilde U(x,r)$,
where $V(x,s)$ is a~neighborhood of $x$, since $\lim_{z\to x} G(z,x)=\infty$.

Finally,  if $V$ is a  neighborhood of~$x$, there exists  $M>0$ with 
$G(\cdot,x)/w\le M$ on~$V^c$. Now let $r:=(CM/w(x))^{-1/\g}$.
Then, for every $y\in \tilde U(x,r)$,
\begin{equation*} 
G(y,x) /w(y)=w(x)\tilde G(y,x) \ge  C\inv  w(x) \tilde \rho(x,y)^{-\g}>M.
\end{equation*} 
Hence $\tilde U(x,r)\subset V$.  
Thus $\tilde \rho$ is a metric for the topology of $X$.
\end{proof}

 We intend to prove the following theorem, where (HJ$'$) trivially holds,
if the measures $\mu_x^U$ are supported by the boundary of $U$.

\begin{theorem}\label{harnack-general}
Suppose that  {\rm (HJ$'$)} holds. Then
  there exist $ \a\in (0,1)$ and $K\ge 1$ such that, for all $x\in X_0$, $0<R<\tilde R_0(x)$
and $h\in \H^+(\tilde U(x,R))$,
\begin{equation}\label{HI-general}
  \sup h(\tilde U(x,\a R))\le K   \inf  h(\tilde U(x,\a R)). 
\end{equation} 
 \end{theorem}

To that end we introduce   normalized  measures
and normalized harmonic functions. 
 For all $x\in X$,  open sets $U$ and closed sets $A$ in $X$, let
\begin{equation}\label{bal-tilde}
                                   \tilde \mu_x^U:=\frac w{w(x)} \, \mx^U, \qquad
            \tilde \ve_x^A:=\tilde \mu_x^{A^c}=\frac w{w(x)} \, \vx^A.
\end{equation} 
For every $U\in \U(X_0)$, let  $\tilde \H(U)$ be 
the set of all universally measurable real  functions~$\tilde h$ on~$X$ 
  such that, for  all  open
sets $V$ with $\ov V\subset U$ and~$x\in V$, the function $\tilde h$ is $\tilde \mu_x^V$-integrable and 
\begin{equation*} 
                     \int \tilde h\,d\tilde \mu_x^V=\tilde h(x).
\end{equation*} 
Obviously,
\begin{equation}\label{tilde-hh}
                                  \tilde \H(U)=\frac 1w\, \H(U).
\end{equation} 

We shall prove that Theorem \ref{G-harnack}   holds for $\tilde \H^+(\tilde U(x,r))$.
Then Theorem~\ref{harnack-general}
follows  easily using (\ref{tilde-hh}), since  $w\le 1$ and $w\ge \lambda$ on $X_0$. 

So let us  verify the assumptions made for Theorem \ref{G-harnack}.
 It is immediately seen  that (M$_0$), (M$_1$) hold for the measures $\tilde \mu_x^U$
(the inequality $\|\tilde \mu_x^U\|\le 1$ is equivalent to (\ref{w-superharmonic})).
Further, for all closed sets $A\subset X_0$, measures  $\nu$ on $A$ and  $x\in X$, 
\begin{equation*}
      G\nu=w\tilde G(w\nu) \und
        \frac{\lambda}{  w(x)} \,\vx^A \le  \tilde \ve_x^A\le \frac1{w(x)}\,  \vx^A. 
\end{equation*} 
Therefore  (G$_1$) holds for the measures $\tilde \ve_x^A$ and $\tilde G$ with
$\tilde  c_1:= \lambda\inv  c_1$ instead of $c_1$ (and measure $\tilde \nu:=w\nu$ instead of $\nu$).
To deal with (G$_3$) we first show the following.

\begin{proposition}\label{tildeG3}
Property {\rm ($\ov G_3$)}  holds for $\tilde G$,   $\tilde \rho$ and $\tilde g(r)=r^{-\g}$.
\end{proposition}

\begin{proof} Let $\tilde c_3:= c_3 C/\lambda^{2}$.
Let $x\in X_0$, $0<r<\tilde R_0(x)$ and $s:=C\inv r^\g$.
By Proposition \ref{tilde-rho},  $V(x,s)\subset \tilde U(x,r)$. 
Let $y\in \tilde U(x,r)^c$.
 Using (\ref{G3prime})  we  see that
\begin{equation*}
\|\tilde \vy^{\ov {\tilde U(x,r)}}\|\ge \frac \lambda {w(y)} \| \, \vy^{\ov{ V(x,s)}}\|
\ge c_3\inv \frac \lambda{w(y)}\, s\, G(y,x)\ge c_3\inv \lambda^2\, s\, \tilde G(y,x)
= \tilde c_3\inv  r^\g \tilde G(y,x).
\end{equation*} 
\end{proof}

\begin{corollary}\label{G3tilde}
Property {\rm ($ G_3$)}  holds for $\tilde G$,  $\tilde \rho$, $\tilde g(r)=r^{-\g}$
and the capacity~$\tkap$ given by $\tilde G$.
\end{corollary} 

\begin{proof} Propositions \ref{tildeG3}          and \ref{amounts}.
\end{proof}

\begin{lemma}\label{tildeHJ}
Suppose that {\rm (HJ$'$)} holds. Then {\rm(HJ)}  holds for $\tilde \rho$ and the measures~$\tilde\mu_x^{\tilde U(x,r)}$, 
$x\in X_0$, $0<r<\tilde R_0(x)$. 
\end{lemma}

\begin{proof} Let $\tilde \a:=\a\b$, $x\in X_0$, $0<r<\tilde R_0(x)$,  
$s:=C\inv r^\g$. Then
\begin{equation}\label{UkV} 
       \tilde      U(x,\tilde \a^k r)\subset \overset\circ V(x, \a^k s), \qquad k\in\nat.
\end{equation} 
Indeed, $\tilde \a^k\le  \b \a^k$, where $\a^{k\g}\le \a^k$, since $\g\ge 1$.
So (\ref{UkV}) follows from the first inclusion in (\ref{UVU}).
Let $y\in \tilde U(x,\tilde \a^2 r)$. Then, by  (\ref{it-bal-1}), (HJ$'$) and (\ref{UVU}), 
\begin{equation*}
  \mu_x^{\tilde U(x,\tilde \a r)}\le       
  \mu_x^{{\overset\circ V(x,\a s)}}\le   c_J \mu_y^{{\overset\circ V(x,s)}}\le c_J \mu_y^{\tilde U(x,r)} \on  \overset\circ V(x,s)^c,
\end{equation*} 
which contains $\tilde U(x,r)^c$.
Hence,  by (\ref{bal-tilde}),  
\begin{equation*} 
\tilde \mu_x^{\tilde U(x,\tilde \a r)}=\frac{w}{w(x)} \,\mu_x^{\tilde U(x,\tilde \a r)}
\le \lambda\inv c_J\,\frac {w}{w(y)} \, \mu_y^{\tilde U(x,r)} 
= \lambda\inv c_J\tilde \mu_y^{\tilde U(x,r)} \on \tilde U(x,r)^c.
\end{equation*} 
 \end{proof} 
 
 As already indicated, Theorem \ref{harnack-general} now follows by an application of~Theorem \ref{G-harnack}
to~$\tilde \H^+(\tilde U(x_0,R))$ and $\tilde \rho$ (recall the identity (\ref{tilde-hh}) and the inequalities
$\lambda\le w\le 1$ on $X_0$).

{\small \noindent 
Wolfhard Hansen,
Fakult\"at f\"ur Mathematik,
Universit\"at Bielefeld,
33501 Bielefeld, Germany, e-mail:
 hansen$@$math.uni-bielefeld.de}\\
{\small \noindent Ivan Netuka,
Charles University,
Faculty of Mathematics and Physics,
Mathematical Institute,
 Sokolovsk\'a 83,
 186 75 Praha 8, Czech Republic, email:
netuka@karlin.mff.cuni.cz}

\end{document}